\documentclass[final,10pt]{amsart}
\usepackage{amsmath,amsfonts,amssymb,amscd,verbatim,comment}
\usepackage{color}
\usepackage{multicol}
\usepackage{bm}
\usepackage{enumitem}

\numberwithin{equation}{section}
\theoremstyle{plain}
 
\newtheorem{theorem}[equation]{Theorem} 

\newtheorem{corollary}[equation]{Corollary} 
\newtheorem{lemma}[equation]{Lemma}
\newtheorem{proposition}[equation]{Proposition}

\theoremstyle{definition}
\newtheorem{definition}[equation]{Definition} 
 
\newtheorem{example}[equation]{Example}

\DeclareMathOperator\ad{ad}
\DeclareMathOperator\af{Aut_{fl}}

\DeclareMathOperator\Aut{Aut}

\DeclareMathOperator\End{End}
\DeclareMathOperator\Ext{Ext}
\DeclareMathOperator\GL{GL}
\DeclareMathOperator\gldim{gldim}
\DeclareMathOperator\GKdim{GKdim}
\DeclareMathOperator\gr{gr}
\DeclareMathOperator\hdet{hdet}
\DeclareMathOperator\id{id}

\DeclareMathOperator\SL{SL}
\DeclareMathOperator\Span{span}

\renewcommand{\det}{\operatorname{det}}
\newcommand{\rt}{t}

\newcommand\CC{\mathbb C}
\newcommand\DD{\mathbb D}
\newcommand\II{\mathbb I}
\newcommand\NN{\mathbb N}
\newcommand\OO{\mathbb O}

\newcommand\TT{\mathbb T}
\newcommand\ZZ{\mathbb Z}

\newcommand\p{\mathsf{p}}

\newcommand\cnt{\mathcal Z}
\renewcommand{\int}{\mathrm{int}}
\newcommand\inv{^{-1}}

\newcommand\iso{\cong}
\newcommand\kk{\Bbbk}

\newcommand\tensor{\otimes}

\newcommand\wa{A_1(\kk)}

\newcommand{\grp}[1]{{\langle {#1} \rangle}}

\renewcommand{\to}{\ensuremath{\longrightarrow}}

\newcommand{\jason}[1]{\textcolor{blue}{#1}}

\title{Fixed rings of generalized Weyl algebras}

\author[Gaddis]{Jason Gaddis}
\address{Miami University, Department of Mathematics, 301 S. Patterson Ave., Oxford, Ohio 45056} 
\email{gaddisj@miamioh.edu}

\author[Won]{Robert Won}
\address{Wake Forest University, Department of Mathematics and Statistics, P. O. Box 7388, Winston-Salem, NC 27109} 
\address{Current address: University of Washington, Department of Mathematics, Box 354350, Seattle, WA 98195}
\email{robwon@uw.edu}

\begin{document}

\begin{abstract} 
We study actions by filtered automorphisms on classical generalized Weyl algebras (GWAs).
In the case of a defining polynomial of degree two,
we prove that the fixed ring under the action of a finite cyclic group of
filtered automorphisms is again a classical GWA, extending a result of Jordan and Wells.
Partial results are provided for the case of higher degree polynomials.
In addition, we establish a version of Auslander's theorem for finite cyclic
groups of filtered automorphisms acting on classical GWAs.
\end{abstract}

\maketitle

\section{Introduction}

Throughout $\kk$ is an algebraically closed field of characteristic zero.
All algebras may be regarded as $\kk$-algebras unless otherwise specified.

The main aim of this paper is to study invariant theory questions related to generalized Weyl algebras.
Generalized Weyl algebras were named by Bavula \cite{B1}
but include many classes of algebras that have been studied in other contexts.
This includes the classical Weyl algebras,
primitive quotients of $U(\mathfrak{sl}_2)$, and ambiskew polynomial rings.

\begin{definition}
Let $D$ be a ring, $\sigma \in \Aut(D)$, and $a \in \cnt(D)$.
The {\sf generalized Weyl algebra (of degree one)} $D[x,y;\sigma,a]$ is the ring
obtained by by adjoining to $D$ the variables $x$ and $y$
subject to the relations
\[ xy = \sigma(a), \quad yx = a, \quad xd=\sigma(d)x, \quad yd=\sigma\inv(d)y,\]
for all $d \in D$.
In the case that $D=\kk[z]$ and $\sigma(z)=z-\alpha$ for some $\alpha \in \kk^\times$
we call $D[x,y;\sigma,a]$ a {\sf classical GWA}. 
\end{definition}

Let $R=D[x,y;\sigma,a]$ be a classical GWA. If $\deg_z(a)=0$,
then $R \iso \kk[x,x\inv]$ and if
$\deg_z(a)=1$, then $R \iso \wa$, the first Weyl algebra over $\kk$.
The study of classical GWAs goes back at least as far as Joseph \cite{joseph}
and were also studied by Hodges under the name
noncommutative deformations of Type-A Kleinian singularities \cite{H}.
Every classical GWA is isomorphic to one where $\sigma(z) = z-1$ and,
by \cite[Theorem 4.2]{BJ}, where $a$ is monic and $0$ is a root of $a$.
We assume these facts throughout without further comment.

The impetus for this study is a result
of Smith stating that $\wa^G \not\iso \wa$ for any nontrivial
finite subgroup $G \subset \Aut(\wa)$ \cite{Sm}.
This relies on an earlier result of Stafford:
if $P$ is a projective right ideal of $\wa$, then
$\End(P) \iso \wa$ if and only if $P$ is cyclic \cite[Theorem 3.1]{St}.
Alev, Hodges, and Velez proved that, for two finite subgroups $G,H \subset \Aut(\wa)$, 
$\wa^G \iso \wa^H$ if and only if $G \iso H$ \cite{AHV}.
Additionally, Alev and Polo extended Smith's theorem to the $n$th Weyl algebra and 
proved a similar result for the universal enveloping algebra of a semisimple Lie algebra \cite{AP}.

A common technique to these papers is reduction modulo primes $p$,
making use of the fact that, over a field of finite characteristic, 
the center of the $n$th Weyl algebra is a polynomial ring.
Unfortunately, this is not the case for classical GWAs with $\deg_z(a) > 1$.
Thus, while we cannot generalize these results entirely, we can
give further insight into the study of the fixed rings of a classical GWA.

Let $R=D[x,y;\sigma,a]$ and let $\lambda$ be a primitive $\ell$th root of unity.
Define the map $\Theta_\lambda:R \rightarrow R$ by
$\Theta_{\lambda}(x)=\lambda x$, $\Theta_{\lambda}(y)=\lambda\inv y$,
and $\Theta_{\lambda}(d)=d$ for all $d \in D$. 
In \cite{JW}, Jordan and Wells prove that $\Theta_\lambda$ is an automorphism
and furthermore that $R^{\grp{\Theta_\lambda}}=D[x^\ell, y^\ell;\sigma^\ell,h_\ell]$ where
\[ h_\ell = \prod_{i=0}^{\ell-1} \sigma^{-i}(a).\]
Suppose $D=\kk[z]$, then it is worth observing that the fixed ring in this case is a classical GWA under the basis $X=x^\ell$, $Y=y^\ell$, and $Z=\frac{1}{\ell}z$.
In \cite{KK}, Kirkman and Kuzmanovich considered fixed rings of GWAs $D[x,y;\sigma,\alpha]$
under automorphisms satisfying
$\phi(x), \phi(y) \in \Span_\kk\{x,y\}$ and $\left.\phi\right|_D \in \Aut(D)$,
as well as their corresponding fixed rings.
In several examples, they show that the fixed ring is again a GWA.

Let $R=\kk[z][x,y;\sigma,a]$ be a classical GWA with
$a = a_0 + \cdots + a_n$, with $a_i \in \kk[z]_i$ and $a_n \neq 0$.
Setting $\deg x = \deg y = n$, and $\deg z = 2$ defines a filtration,
called the {\sf standard filtration}, on $R$.
Unless otherwise noted, we assume throughout that this is the filtration on $R$.
Under the standard filtration, $\gr R \iso \kk[x,y,z]/(yx-a_n)$, a complete intersection domain.
Our interest is in {\sf filtered automorphisms}, i.e., maps $\phi \in \Aut(R)$ such that $\deg(\phi(r)) = \phi(\deg(r))$ for all $r \in R$.
As $R$ is assumed to be a classical GWA, $\phi$ is a filtered automorphism
if and only if this holds for $r=x$, $y$, and $z$.
We denote the group of filtered automorphisms of $R$ by $\af(R)$.
In Section \ref{sec.filt}, we determine $\af(R)$.

In Theorem \ref{thm.fixed}, we show that when $\deg_z(a) = 2$ and $g \in \af(R)$ is of finite order, then $R^{\grp{g}}$ is a classical GWA.
We view this as a step toward a Shephard-Todd-Chevalley theorem for classical GWAs,
but do not yet know a good analog for a reflection group in this setting.
Also, one obtains a version of Smith's theorem (Corollary \ref{cor.smith}).
It is then reasonable to conjecture that an analog of Smith's theorem is true 
for finite groups of filtered automorphisms acting on a classical GWA with $\deg_z(a)=2$.
When $\deg_z(a) > 2$, the group $\af(R)$ is much more restricted
(Theorems \ref{thm.ngeq3} and \ref{thm.refgrp}).
In this case we are able to compute the fixed ring of $R$ under the action of any filtered
automorphism $g$ of finite order (Theorem \ref{thm.n3omega}). 
However, there are certain cases in which we can not detect whether or not $R^{\grp{g}}$ is a GWA.

We are also interested in the homological determinant 
and its connection to the above results.
The homological determinant of a linear automorphism of a noncommutative algebra 
generalizes the notion of the determinant of a linear map---
in fact, when applied to a commutative polynomial ring,
the homological determinant restricts to the usual determinant.

We refer the reader to \cite{gourmet} for a full definition of the homological determinant,
but we note an essential result that will be important for our analysis.
Let $A$ be a filtered, noetherian, AS-Gorenstein ring such that $\gr A$ is commutative and let $g \in \af(A)$.
By \cite[Lemma 2.1 and Proposition 2.4]{JZ},
one may define the homological determinant of $g$ to be 
\[\hdet_A(g) = \hdet_{\gr A} (g) = \det_{\gr A} (g).\]
Using this result, we prove that all filtered automorphisms of classical GWAs act with homological determinant 1
(Theorems \ref{thm.ngeq3} and \ref{thm.flaut}).
We denote by $\SL(A)$ the subgroup of $\Aut(A)$ consisting of automorphisms of homological determinant 1.

A noetherian ring $A$ of finite injective dimension is called {\sf Auslander-Gorenstein}
if for any (left or right) module $M$ and submodule $N$ of $\Ext_A^s(M,A)$, $s \in \ZZ_+$,
we have $\Ext_A^i(N,A)=0$ for $i<s$.
By \cite[Theorem 2.2]{H}, every classical GWA $R=\kk[z][x,y;\sigma,a]$ 
with $\deg_z(a) \geq 2$ is Auslander-Gorenstein.
If $R$ is a classical GWA and $G$ is a finite subgroup of $\af(R)$,
then $G \subset \SL(R)$ and so $R^G$ is filtered Auslander-Gorenstein by \cite[Theorem 3.5]{JZ},
that is, $\gr(R^G)$ is Auslander-Gorenstein.
This will be useful in the case of $\deg_z(a) \geq 3$ when we are not able to determine whether or not
$R^\grp{g}$ is a GWA for all $g \in \af(R)$.
We also recover this result as a consequence of Theorem \ref{thm.fixed}
in the case of a classical GWA with $\deg_z(a)=2$ and $G \subset \af(R)$ a finite cyclic group.
Though this implies that $R^G$ has finite left and right injective dimension,
we will see that this is not enough to guarantee finite global dimension (Corollary \ref{cor.gldim}).

We end in Section \ref{sec.auslander} with a note regarding Auslander's theorem for GWAs.
Given an algebra $A$ and a group $G$ acting as automorphisms on $A$,
the {\sf skew group algebra} $A\# G$ is defined to be the $\kk$-vector space $A \tensor \kk G$ with multiplication,
\[ (a \# g)(b \# h) = a g(b) \# gh \quad\text{for all $a,b \in A$, $g,h \in G$.}\]
For a filtered algebra $A$ and a finite group $G$ acting as filtered automorphisms on $A$, 
the {\sf Auslander map} is given by
\begin{eqnarray*}
\gamma_{A,G} : A \# G & \to & \End_{A^G}(A) \\
       a \# g & \mapsto & 
       \left(\begin{matrix}A & \to & A \\ b & \mapsto & ag(b)\end{matrix}\right).
\end{eqnarray*}
If $G$ is a finite group that contains no reflections acting linearly on $A=\kk[x_1,\hdots,x_p]$,
then a theorem of Auslander asserts that $\gamma_{A,G}$ is an isomorphism \cite{A}.
In this setting, Auslander's theorem is a sort of dual result to Shephard-Todd-Chevalley.
However, this is not the case for filtered actions on classical GWAs.
That is, there are classical GWAs $R$ and finite groups $G \subset \af(R)$
such that $R^G$ is again a classical GWA and for which $\gamma_{R,G}$ is an isomorphism.
For example, if $G$ is a finite group acting linearly on the first Weyl algebra $\wa$,
then the action of $G$ is outer and so $\gamma_{\wa,G}$ is an isomorphism \cite[Theorems 2.4]{Mo} and if $G$ is cyclic, 
then $\wa^G$ is a classical GWA (Proposition \ref{prop.wa}).
Similarly, since the units of a classical GWA $R$ all live in the
degree zero component (under the $\ZZ$-grading), $\kk[z]$, the units of $R$ are just $\kk^\times$.
Thus, every finite group action on $R$ is outer and so if $R$ is simple, then we may apply the same theorem.
We present another method that will include classical GWAs that are not simple.

Let $A$ be an affine algebra generated in degree 1 and $G$ a finite 
subgroup of $\GL_n(\kk)$ acting on $A_1$.
The {\sf pertinency} of the $G$-action on $A$ is defined to be
\[ \p(A,G) = \GKdim A - \GKdim (A\# G)/(f_G)\]
where $(f_G)$ is the two sided ideal of $A\#G$ generated by
$f_G = \sum_{g \in G} 1\# g$ and $\GKdim$ is the Gelfand-Kirillov (GK) dimension.
The notion of pertinency was developed by Bao, He, and Zhang as a way
to study the Auslander map for noncommutative algebras \cite{BHZ2,BHZ1}.
It is possible to define pertinency in terms of any dimension function on right $A$-modules,
but GK dimension is sufficient for our purpose.
Under suitable conditions, the Auslander map is an isomorphism for $(A,G)$ if and only if $\p(A,G) \geq 2$.
We show that this holds for $(R,G)$ where $R$ is a classical GWA $R$ and 
$G \subset \af(R)$ is a finite cyclic group (Theorem \ref{thm.auslander}).

\section{Group actions preserving the standard filtration}
\label{sec.filt}

Throughout this section, assume $R=\kk[z][x,y;\sigma,a]$ is a classical GWA.
Our primary goal will be to compute $\af(R)$ and prove that
$R^{\grp{g}}$ is a classical GWA for every $g \in \af(R)$ with $|g|<\infty$.
Although our interest is primarily in fixed rings of higher degree classical GWAs,
as a warm-up we compute the fixed rings for cyclic subgroups of filtered automorphisms acting on the first Weyl algebra, $\wa$.
\begin{proposition}
\label{prop.wa}
Let $g \in \af(\wa)$ have finite order, then the fixed ring $\wa^{\grp{g}}$ is a classical GWA.
\end{proposition}

\begin{proof}
A filtered map $g: \wa \to \wa$ given by
\[g(x)=a_1x + a_2y + a_3, \quad g(y)=b_1x+b_2y+b_3,\]
for some $a_i,b_j \in \kk$, is an automorphism if and only if $a_1b_2-a_2b_1 = 1$.
Assume first that $b_2-a_1\pm\sqrt{w^2-4}\neq 0$.
We diagonalize the action by setting $w=a_1+b_2$ and
\begin{align*}
X &= a\left( \left(\frac{2b_1a_2(w-2+\sqrt{w^2-4})}{b_2-a_1+\sqrt{w^2-4}}\right)x 
	+ (w-2+\sqrt{w^2-4})y
    + ((a_1-b_2)a_3+2b_3a_2+a_3\sqrt{w^2-4})\right), \\
Y &= b\left( \left(\frac{2b_1a_2(w-2-\sqrt{w^2-4})}{b_2-a_1-\sqrt{w^2-4}}\right)x 
	+ (w-2-\sqrt{w^2-4})y
    + ((a_1-b_2)a_3+2b_3a_2-a_3\sqrt{w^2-4})\right),
\end{align*}
for any $a,b \in \kk^\times$.
Then $X,Y$ is a basis for $\wa$ and we may choose $a,b$ such that $XY-YX+1=0$. 
Moreover, if we let $\beta = \frac{1}{2}(w+\sqrt{w^2-4})$, then one can check that
$g(X)=\beta X$ and $g(Y)=\beta Y$, so $|\beta|=\ell<\infty$.
Set $Z=YX$. By \cite[Theorem 2.7]{JW}, 
$\wa^\grp{g}=\kk[Z][X^\ell,Y^\ell;\varsigma^\ell, A(Z)]$ 
where $\varsigma:\kk[Z]\rightarrow\kk[Z]$ is given by $\varsigma(Z)=Z-1$ and
$A(Z) = \prod_{i=0}^{\ell-1} \varsigma^{-i}(a'(Z))$.

In the case that $b_2-a_1\pm\sqrt{w^2-4}=0$ we have $b_2 = a_1\inv$, hence $b_1a_2=0$.
In this case the analysis simplifies significantly.
Assuming $b_1=0$ (the case $a_2=0$ is similar), we may take
\begin{align*}
X &= a\left( (a_1-1)^3(a_1+1)x + (a_1-1)a_1a_2y + (a_1^2a_3 + a_1a_2b_3-a_3)\right), \\
Y &= b\left( (a_1-1)y - a_1b_3\right),
\end{align*}
for any $a,b \in \kk^\times$.
Again, $X,Y$ is a basis for $\wa$ and we may choose $a,b$ such that $XY-YX+1=0$. 
Here we set $\beta=a_1$.
\end{proof}

\begin{comment}
In particular, they are all conjugate to one of the 
special subgroups of $\SL_2(\kk)$ listed below.
\begin{itemize}
\item[] $\CC$ : a cyclic group of order $n$.
\item[] $\DD$: a binary dihedral group of order $4n$.
\item[] $\TT$ : a binary tetrahedral group of order $24$.
\item[] $\OO$ : a binary octahedral group of order $48$.
\item[] $\II$ : a binary icosahedral group of order $120$.
\end{itemize}
\end{comment}

Let $n = \deg_z(a)$, $\lambda \in \kk$, $m \in \NN$, and let $\Delta_m$ be the linear map $\kk[z] \rightarrow \kk[z]$ given by $\sigma^{m}-1$. 
Bavula and Jordan \cite{BJ} define the following maps.
\begin{align*}
&\Theta_\lambda: x \mapsto \lambda x, y \mapsto \lambda\inv y, z \mapsto z, \\
&\Psi_{m,\lambda}: 
	x \mapsto x, 
    y \mapsto y + \sum_{i=1}^n \frac{\lambda^i}{i!} \Delta_m^i(a)x^{im-1},
    z \mapsto z-m\lambda x^m, \\
&\Phi_{m,\lambda}: 
	x \mapsto x + \sum_{i=1}^n \frac{(-\lambda)^i}{i!} y^{im-1}\Delta_m^i(a),
	y \mapsto y,
    z \mapsto z+m\lambda y^m.
\end{align*}
Note that $\Psi_{0,\lambda}$ and $\Phi_{0,\lambda}$ are both the identity map.
If there exists some $\rho \in \kk$ such that $a(\rho-z)=(-1)^n a(z)$, then $a$ is said to be {\sf reflective}.
By \cite[Theorem 3.29]{BJ}, $\Aut(R)$ is generated by $\Theta_{\lambda}$, $\Psi_{m, \lambda}$, and $\Phi_{m,\lambda}$, $\lambda \in \kk^\times$ and $m \in \NN$, when 
$a$ is not reflective.
On the other hand, when $a$ is reflective, then $\Aut(R)$ has an additional generator $\Omega$ given by
\[ \Omega(x) = y, \quad \Omega(y) = (-1)^nx, \quad \Omega(z) = 1+\rho-z.\]
Below, we consider some relations between the generators of $\Aut(R)$.

\begin{proposition}
\label{prop.gmaps}
Let $R$ be a classical GWA. The following relations hold in $\Aut(R)$. For all $m \in \NN$ and all $\beta, \lambda, \mu \in \kk$,
\begin{enumerate}
\item $\Phi_{m,\mu} \circ \Phi_{m,\lambda} = \Phi_{m,\mu+\lambda}$
and $\Psi_{m,\mu} \circ \Psi_{m,\lambda} = \Psi_{m,\mu+\lambda}$, 
\item $\Theta_{\beta} \circ \Phi_{m,\lambda} = \Phi_{m, \lambda\beta^{-m}} \circ \Theta_{\beta}$ and 
$\Theta_{\beta} \circ \Psi_{m,\lambda} = \Psi_{m, \lambda\beta^m} \circ \Theta_{\beta}$, and
\item $\Theta_{\beta} \circ \Theta_{\gamma} = \Theta_{\beta \gamma}$.
\end{enumerate}
When $a$ is reflective we have the following additional relations. For all $m \in \NN$ and all $\beta, \lambda \in \kk$,
\begin{enumerate}[resume]
\item $\Omega \circ \Theta_{\beta} = \Theta_{\beta\inv} \circ \Omega$ and
\item $\Phi_{m,\lambda} \circ \Omega = \Omega \circ \Psi_{m,\lambda}$.
\end{enumerate}
\end{proposition}

\begin{proof}
By \cite{BJ}, the maps $\ad x^m$ and $\ad y^m$ are locally nilpotent derivations
of $R$ and hence $\Phi_{m,\lambda} = e^{\lambda \ad x^m}$ are automorphisms of $R$.
It then follows easily that 
\[ 
\Phi_{m,\mu} \circ \Phi_{m,\lambda} 
	= e^{\mu \ad x^m} \circ e^{\lambda \ad x^m}
    = e^{(\mu+\lambda)\ad x^m}
    = \Phi_{m,\mu+\lambda}.
\]
The claim for the maps $\Psi_{m,\lambda}$ is similar. Thus, (1) holds.

We will check the first claim of (2) by verifying that the relation holds on the generators.
\begin{align*}
\Theta_{\beta}(\Phi_{m,\lambda}(x))
	&= \Theta_\beta\left(x + \sum_{i=1}^n \frac{(-\lambda)^i}{i!} y^{im-1}\Delta_m^i(a)\right) \\
    &= \beta x + \sum_{i=1}^n \frac{(-\lambda)^i}{i!} \beta^{1-im} y^{im-1}\Delta_m^i(a) \\
    &= \beta x + \sum_{i=1}^n \frac{(-\lambda\beta^{-m})^i}{i!} \beta y^{im-1}\Delta_m^i(a) \\
    &= \Phi_{m,\lambda\beta^{-m}}(\beta x) 
    = \Phi_{m,\lambda\beta^{-m}}(\Theta_\beta(x))\\
\Theta_{\beta}(\Phi_{m,\lambda}(y))
	&= \Theta_\beta(y) = \beta\inv y
    = \Phi_{m,\lambda\beta^{-m}}(\Theta_\beta(y)) \\
\Theta_{\beta}(\Phi_{m,\lambda}(z))
	&= \Theta_\beta(z+m\lambda y^m)
    = z+m\lambda (\beta\inv y)^m
    = z+m(\lambda\beta^{-m}) y^m \\
    &= \Phi_{m,\lambda\beta^{-m}}(z)
    = \Phi_{m,\lambda\beta^{-m}}(\Theta_\beta(z)).
\end{align*}
Thus, $\Theta_{\beta} \circ \Phi_{m,\lambda} = \Phi_{m, \lambda\beta^{-m}} \circ \Theta_{\beta}$ as claimed. The second relation in (2) holds similarly.

We leave the claims in (3) and (4) to the reader and finish by checking (5).
Assume that $a$ is reflective.
By \cite[Equations (7) and (9)]{BJ} we have
\[
(\ad x^m)^i(y) = \Delta_m^i(a)x^{im-1} \quad\text{and}\quad
(\ad y^m)^i(x) = (-1)^iy^{im-1}\Delta_m^i(a).
\]
Thus,
\[ \Omega((\ad x^m)^i(y)) = (-1)^n (\ad y^m)^i(x).\]
Using this, we check that the relation holds on the generators of $R$.
\begin{align*}
\Phi_{m,\lambda}(\Omega(x))
	&= \Phi_{m,\lambda}(y) = y
    = \Omega(x) = \Omega(\Psi_{m,\lambda}(x)) \\
\Phi_{m,\lambda}(\Omega(y))
	&= \Phi_{m,\lambda}((-1)^n x) 
    = (-1)^n \left( x + \sum_{i=1}^n \frac{(-\lambda)^i}{i!} y^{im-1}\Delta_m^i(a)\right)  \\
    &= (-1)^n \left( x + \sum_{i=1}^n \frac{\lambda^i}{i!} (\ad y^m)^i(x) \right)
    = \Omega\left( y + \sum_{i=1}^n \frac{\lambda^i}{i!} (\ad x^m)^i(y) \right) \\
    &= \Omega\left( y + \sum_{i=1}^n \frac{\lambda^i}{i!} \Delta_m^i(a)x^{im-1} \right) 
    = \Omega(\Psi_{m,\lambda}(y)) \\
\Phi_{m,\lambda}(\Omega(z))
	&= \Phi_{m,\lambda}(1+\rho-z)
    = 1+\rho-(z+m\lambda y^m) \\
    &= (1+\rho-z)-m\lambda y^m 
    = \Omega(z-m\lambda x^m)
    = \Omega(\Psi_{m,\lambda}(z)).
\end{align*}
Hence, $\Phi_{m,\lambda} \circ \Omega = \Omega \circ \Psi_{m,\lambda}$ as claimed.
\end{proof}

We next give criteria for identifying filtered automorphisms based on the action on $z$.
This will allow us to completely determine $\af(R)$ when $R$ is a classical GWA with $\deg_z(a)>2$.
It will also be a useful step in the case of $\deg_z(a)=2$.

\begin{lemma}
\label{lem.fil1}
Let $\phi$ be a filtered automorphism of a classical GWA $R=\kk[z][x,y;\sigma,a]$.
\begin{enumerate}
\item If $\phi(z)=k z+c$ for some $c \in \kk$ and $k \in \kk^\times$, then either 
\begin{itemize}
\item $\phi=\Theta_\beta$ for some $\beta \in \kk^\times$; or
\item $a$ is reflective and $\phi=\Omega$ or $\phi=\Omega \circ \Theta_{-1}$.
\end{itemize}
\item If $\phi(z) \neq k z+c$ for some $c \in \kk$ and $k \in \kk^\times$, 
then $\deg_z(a) \leq 2$.
\end{enumerate}
\end{lemma}

\begin{proof}
(1) Suppose $\phi(z)=k z+c$ for some $c \in \kk$ and $k \in \kk^\times$. 
As $\phi$ is a filtered map, we may write
\begin{align*}
\phi(x) &= k_{11} x + k_{12} y + p_1(z), \\
\phi(y) &= k_{21} x + k_{22} y + p_2(z),
\end{align*}
where $k_{ij} \in \kk$ and $p_i(z)$ are polynomials in $z$ of degree at most  $n/2$. Then
\begin{align*}
0 	&= \phi([x,z]+x) \\
	&= [k_{11} x + k_{12} y + p_1(z),kz+c] + (k_{11} x + k_{12} y + p_1(z)) \\
    &= k_{11} (k[x,z]+x) + k_{12}(k[y,z]+y) - p_1(z) \\
    &= k_{11} (1-k)x + k_{12}(1+k)y - p_1(z).
\end{align*}
A similar computation shows that
\[ 0 = \phi([y,z]-y) = -k_{21}(1+k)x -k_{22}(1-k)y - p_2(z).\]
If $k_{11}=k_{12}=0$, then $\phi(x) \in \kk[z]$, violating the surjectivity of $\phi$.
Similarly, we may not have $k_{21}=k_{22}=0$.
If $k_{12}=0$, then $k_{11} \neq 0$ and $k=1$, so $k_{21}=0$.
Otherwise, $k_{11}=0$ so $k=-1$ and $k_{22}=0$.
In either case, $p_1(z)=p_2(z)=0$.

In the first case,
\[ 0 = \phi(yx - a(z)) = k_{11}k_{22} a(z) - a(z+c). \]
We may assume without loss of generality that 
$a(z) = z(z-\rt_1)\cdots(z-\rt_{n-1})$ for some $\rt_i \in \kk$.
Thus, $k_{22}=k_{11}\inv$ and $c=0$, so $\phi=\Theta_{k_{11}}$.

In the second case,
\[ 0 = \phi(yx - a(z)) = k_{12}k_{21} a(z-1) - a(-z+c). \]
As $a(z-1)$ is monic and the leading coefficient of $a(-z+c)$ is $(-1)^n$,
then $k_{12}k_{21} = (-1)^n$.
It follows that $a$ is reflective and $\phi=\Omega$ or $\phi=\Omega \circ \Theta_{-1}$.

(2) If $\phi \in \af{R}$, then $\phi(z) = kz + p(x,y)$ for some polynomial $p$
in $x$ and $y$.
In the filtration, $\deg(x)=\deg(y)=\deg(a)=n$, but $\deg(z)=2$.
By the hypothesis and part (1), we must have $\deg(p(x,y)) \geq 1$, whence $n=2$.
\end{proof}

\begin{theorem}
\label{thm.ngeq3}
Suppose $R=\kk[z][x,y;\sigma,a]$ is a classical GWA with $\deg_z(a) > 2$.
If $a$ is not reflective, then $\af(R)$ is generated by the maps $\Theta_\lambda$.
If $a$ is reflective, then $\af(R)$ is generated by $\Omega$ and the maps $\Theta_\lambda$.
Moreover, we have $\hdet g = 1$ for all $g \in \af(R)$.
\end{theorem}

\begin{proof}
This follows almost entirely from Lemma \ref{lem.fil1}.
Let $g \in \af(R)$ and recall that we have 
$\gr R = \kk[x,y,z]/(xy-a_n)$.
By the discussion in the introduction and a routine check,
$\hdet_R(g) = \hdet_{\gr R}(g) = \det_{\kk[x,y,z]}(g) = 1$.
\end{proof}

In Theorem \ref{thm.refgrp} we completely determine the finite subgroups of $\af(R)$ 
in the case when $n > 3$.

Assume $n = 2$. Without loss of generality, $a = z(z-\rt)$ for some $\rt \in \kk$. Then $\Delta(a) = -2z + \rt + 1$ and $\Delta^2(a) = 2$.
The generators of $\Aut(R)$ above that are also filtered maps can be stated explicitly.
\begin{align*}
&\Theta_\lambda: &
	x &\mapsto \lambda x, &
    y &\mapsto \lambda\inv y, &
    z &\mapsto z, \\
&\Psi_{1,\lambda}: &
	x &\mapsto x, &
    y &\mapsto y - 2\lambda z + \lambda^2  x + \lambda(\rt +1), &
    z &\mapsto z-\lambda x,\\
&\Phi_{1,\lambda}: &
	x &\mapsto x + 2\lambda z + \lambda^2 y - \lambda(\rt +1), &
	y &\mapsto y, &
    z &\mapsto z+\lambda y, \\
&\Omega: &
	x &\mapsto y, & 
    y &\mapsto x, &
    z &\mapsto 1+\rt-z.
\end{align*}
Let $G$ be the group generated by these automorphisms.
We will show below that $\af(R)=G$ in this case.

Before proving our main result regarding $\af(R)$,
we need one more technical lemma.

\begin{lemma}
\label{lem.flaut}
Let $R=\kk[z][x,y;\sigma,a]$ be a classical GWA with $\deg_z(a)=2$.
Suppose $\Gamma \in \af(R)$ and
\[\Gamma(z)=k_1 x + k_2 y + k_4\] 
for some $k_i \in \kk$. Then $k_1k_2 \neq 0$.
\end{lemma}

\begin{proof}
It is clear that we may not have $k_1=k_2=0$.
Suppose $k_2=0$. The case $k_1=0$ follows similarly. Write
\begin{align*}
\Gamma(x) &= \ell_1 x + \ell_2 y + \ell_3 z + \ell_4, \\
\Gamma(y) &= m_1 x + m_2 y + m_3 z + m_4, 
\end{align*}
for $\ell_i,m_i \in \kk$. Then
\begin{align*}
\Gamma(x)
    &= [\Gamma(z),\Gamma(x)]
    = [k_1 x + k_4,\ell_1 x + \ell_2 y + \ell_3 z + \ell_4]
    = [k_1 x,\ell_2 y + \ell_3 z] \\
\Gamma(y)
	&= [\Gamma(y),\Gamma(z)]
    = [m_1 x + m_2 y + m_3 z + m_4,k_1x+k_4]
    = [m_2 y + m_3 z,k_1x].
\end{align*}
In both cases, the image of the commutator is in the subalgebra generated
by $x$ and $z$, implying $\ell_2=m_2=0$.
This contradicts the surjectivity of $\Gamma$.
\end{proof}

\begin{theorem}
\label{thm.flaut}
Let $R=\kk[z][x,y;\sigma,a]$ be a classical GWA with $\deg_z(a)=2$.
Then $\af(R)=G$.
Moreover, if $g \in \af(R)$, then $\hdet(g)=1$.
\end{theorem}

\begin{proof}
The statement on homological determinant follows analogously to Theorem \ref{thm.ngeq3} once we have shown that $\af(R)=G$.
Clearly, $G \subset \af(R)$.
Let $\Gamma \in \af(R)$. We may write
$\Gamma(z)=k_1 x + k_2 y + k_3 z + k_4$ for some $k_i \in \kk$.

If $k_1=k_2=0$, then $\Gamma \in G$ by Lemma \ref{lem.fil1}.
Suppose that $k_1=0$ but $k_2 \neq 0$.
By Lemma \ref{lem.flaut}, $k_3\neq 0$. Then
\[ \Phi_{1,-k_2/k_3}(\Gamma(z)) 
	= k_2 y + k_3\left( z - (k_2/k_3)y \right) + k_4
    = k_3z + k_4.
\]
Thus, $\Phi_{1,-k_2/k_3}\circ\Gamma \in G$ again by Lemma \ref{lem.fil1},
so $\Gamma \in G$.
Similarly, if $k_2=0$ but $k_1 \neq 0$, then
\[ \Psi_{1,k_1/k_3}(\Gamma(z)) 
	= k_1 x + k_3\left( z - (k_1/k_3)x \right) + k_4
    = k_3 z + k_4.
\]

Finally, suppose $k_1,k_2 \neq 0$.
Note that we may have $k_3=0$ in this case.
Set $\lambda$ to be a root of $k_1\lambda^2+k_3\lambda+k_2=0$. Then
\begin{align*}
\Phi_{1,\lambda}(\Gamma(z))
	&= k_1 (x + 2\lambda z + \lambda^2 y - \lambda(\rt+1))
    	+ k_2 y + k_3 (z+\lambda y) + k_4 \\
	&= k_1 x + (k_1\lambda^2+k_3\lambda+k_2)y 
    	+ (2k_1\lambda + k_3)z + (k_4-k_1\lambda(\rt+1))  \\
    &= k_1 x + (2k_1\lambda + k_3)z + (k_4-k_1\lambda(\rt+1)).
\end{align*}
Note that $2k_1\lambda+k_3 \neq 0$ by Lemma \ref{lem.flaut}.
We now defer to the above computation.
\end{proof}

Using the techniques of Theorem \ref{thm.flaut}, or straightforward computation,
we achieve our last relation between the generators of $G$.
Given $\lambda,\mu \in \kk$, set $\eta = 1-\lambda\mu$. Then
\begin{align}
\label{eq.phipsi}
\Phi_{1,\mu} \circ \Psi_{1,\lambda}
	= \Psi_{1,\lambda\eta\inv} \circ \Phi_{1,\mu\eta} \circ \Theta_{\eta^{-2}}.
\end{align}

For $\lambda,\mu \in \kk$ and $\beta \in \kk^\times$, set 
$\tau_{\lambda,\mu,\beta}=\Psi_{1,\lambda} \circ \Phi_{1,\mu} \circ \Theta_\beta$.
These maps satisfy
\begin{align*}
\tau_{\lambda,\mu,\beta}: \quad &x \mapsto \beta((\lambda\mu-1)^2 x + \mu^2 y + 2\mu(1-\lambda\mu)z + \mu(\lambda\mu-1)(\rt+1)) \\
 &y \mapsto \beta\inv(y+\lambda^2 x-2\lambda z + \lambda(\rt+1)) \\
 &z \mapsto (1-2\lambda\mu)z + \lambda(\lambda\mu-1)x + \mu y + \lambda\mu(\rt+1).
\end{align*}
Note that $\tau_{\lambda,\mu,\beta}=\id$ if and only if $\lambda=\mu=0$ and $\beta=1$.

\begin{corollary}
\label{cor.filt}
Let $R=\kk[z][x,y;\sigma,a]$ be classical GWA with $\deg_z(a)=2$.
If $g \in \af(R)$, then either $g=\tau_{\lambda,\mu,\beta}$ or $\tau_{\lambda,\mu,\beta} \circ \Omega$ for an appropriate choice of $\lambda,\mu,\beta$.
\end{corollary}

\begin{proof}
This follows from \eqref{eq.phipsi}, Proposition \ref{prop.gmaps}, and Theorem \ref{thm.flaut}.
\end{proof}

Now that we understand $\af(R)$ when $\deg_z(a)=2$, we are ready
to consider fixed rings of $R$ by its cyclic subgroups.
We first give two examples that illustrate our methods before stating our main theorem.

\begin{example}
\label{ex.fixed}
Let $R=\kk[z][x,y;\sigma,a]$ be a classical GWA with $a = z(z-\rt)$.
Let $\alpha \in \kk$ and $\beta \in \kk^\times$.

(1) This example is similar to \cite[Example 2.7]{KK}, in which the authors compute $\wa^\grp{\Omega}$.
Define
\[
X = \frac{i}{2}(x-y)-\left(z-\frac{1+\rt}{2}\right), \quad
Y = \frac{i}{2}(x-y)+\left(z-\frac{1+\rt}{2}\right), \quad
\text{ and} \quad  
Z = \frac{i}{2}(x+y).
\]
Set $k=\frac{1}{4}(1-\rt^2)$,
$K_{\pm}=\frac{1}{2}(-1\pm\sqrt{1-4k})$, and $a'(Z) = (Z-K_+)(Z-K_-)$. 
Let $\varsigma: \kk[Z] \to \kk[Z]$ be the automorphism mapping $Z$ to $Z-1$. 
Then $R=\kk[Z][X,Y;\sigma, a'(Z)]$ and we have $\Omega(X)=-X$, $\Omega(Y)=-Y$, $\Omega(Z)=Z$.
Thus, by \cite[Theorem 2.7]{JW}, $R^\grp{\Omega}=\kk[Z][X^2,Y^2;\varsigma^2, a'(Z)\varsigma\inv(a'(Z))]$.

(2) Let $\beta$ be a primitive $\ell$th root of unity for some $\ell \geq 2$. 
Set $\pi_{\alpha, \beta} = \Phi_{1,\alpha} \circ \Theta_{\beta}$ and note that 
$|\pi_{\alpha,\beta}|=\ell$ by Proposition \ref{prop.gmaps}. Define
\[
X = x + \frac{\alpha^2\beta^2}{(\beta-1)^2}y + \frac{2 \alpha \beta}{ \beta - 1} z - \frac{\alpha \beta (r+1)}{\beta-1}, \quad
Y = y, \quad
\text{ and} \quad  
Z =  z + \frac{\alpha \beta}{\beta-1}y.
\]
Let $\varsigma: \kk[Z] \to \kk[Z]$ be the automorphism mapping $Z$ to $Z-1$. 
Then $R=\kk[Z][X,Y;\varsigma, a(Z)]$ and we have $\pi_{\alpha, \beta}(X) = \beta X$, 
$\pi_{\alpha, \beta}(Y) = \beta\inv Y$ and $\pi_{\alpha,\beta} (Z) = Z$.
Thus, $R^\grp{\pi_{\alpha,\beta}}=\kk[Z][X^\ell,Y^\ell;\varsigma^\ell, A(Z)]$ 
where $A(Z) = \prod_{i=0}^{\ell-1} \varsigma^{-i}(a(Z))$.
\end{example}

We now show that, given an appropriate basis, one can diagonalize the action of $\tau_{\lambda,\mu,\beta}$ 
and $\tau_{\lambda,\mu,\beta} \circ \Omega$ when the maps have finite order.
Computations for the next theorem were done using Maple and the NCAlgebra package of Macaulay2.

\begin{theorem}
\label{thm.fixed}
Let $R=\kk[z][x,y;\sigma,a]$ be a classical GWA with $a=z(z-\rt)$.
If $g \in \af(R)$ with $|g|=\ell$, $2\leq \ell <\infty$, then the 
action of $g$ is diagonalizable and hence $R^\grp{g}$ is again a GWA.
\end{theorem}

\begin{proof}
First suppose that $g=\tau_{\lambda,\mu,\beta}$. 
Set $w=\beta\lambda\mu-\beta-1$ and
\[
A = (2\sqrt{\lambda\mu\beta(w^2-4\beta)})\inv, \qquad 
C = (-1\sqrt{w^2-4\beta})\inv, \qquad
K_{\pm}=\frac{(\rt +1)(w+2)}{2\sqrt{w^2-4\beta}}\pm \frac{\rt -1}{2}.
\]
Let
\begin{align*}
X &= A\left(
		\lambda \left(w+2+\sqrt{w^2-4\beta} \right) x
	+ 	\beta\mu \left(w+2-\sqrt{w^2-4\beta} \right) y 
    - 	4\lambda\mu\beta z + 2\lambda\mu\beta(\rt +1)
    \right) \\
Y &= A\left(
		\lambda \left(w+2-\sqrt{w^2-4\beta} \right) x
	+ 	\beta\mu \left(w+2+\sqrt{w^2-4\beta} \right) y 
    - 	4\lambda\mu\beta z + 2\lambda\mu\beta(\rt +1)
    \right) \\
Z &= C\left( \lambda x + \beta\mu y - (w+2)z\right).
\end{align*}

Next, suppose $g=\tau_{\lambda,\mu,\beta} \circ \Omega$. 
In the generic case, $\lambda\mu \neq 1$, set $w=\lambda+\mu\beta$ and
\[
A=\frac{\sqrt{\beta(1-\lambda\mu)}}{2w\sqrt{w^2-4\beta}}, \qquad 
C=-\frac{1}{w^2-4\beta}, \qquad
K_{\pm}=\frac{(\rt +1)(\lambda-\beta\mu)}{2\sqrt{w^2-4\beta}}\pm \frac{\rt -1}{2}.
\]
Let
\begin{align*}
X &= A\left(\left( \lambda^2-(\mu\beta)^2 + w\sqrt{w^2-4\beta} \right) x
	+ \left( \frac{(\mu\beta)^2-\lambda^2+ w\sqrt{w^2-4\beta}}{\beta(\lambda\mu-1)} \right)y - 4wz + 2w(\rt +1)\right) \\
Y &= A\left(\left( \lambda^2-(\mu\beta)^2 - w\sqrt{w^2-4\beta} \right) x
	+ \left( \frac{(\mu\beta)^2-\lambda^2- w\sqrt{w^2-4\beta}}{\beta(\lambda\mu-1)} \right)y - 4wz + 2w(\rt +1)\right) \\
Z &= C\left( -\beta(\lambda\mu-1)x + y + (\beta\mu-\lambda)z\right).
\end{align*}
We consider the special case when $\lambda\mu=1$ at the end.

Let $X, Y, Z \in R$ be defined as above depending on the case. 
In either case, let $\varsigma: \kk[Z] \to \kk[Z]$ be the automorphism mapping 
$Z$ to $Z-1$ and set $a'(Z):=(Z-K_+)(Z-K_-)$.
Direct computations show that 
$XZ = (Z - 1) X$, $YZ = (Z + 1) Y$, $YX = a'(Z)$, and $XY=\varsigma(a'(Z))$. 
Since $X$, $Y$, and $Z$ generate $R$ as an algebra, 
then $R$ has a presentation as the classical GWA $\kk[Z][X,Y;\varsigma, a'(Z)]$.

Let
\[ \gamma = \frac{1}{2\beta}\left(w^2-2\beta+w\sqrt{w^2-4\beta}\right)\]
and note that if $w^2-4\beta=0$, then $\gamma=1$.
A check shows that 
$\tau_{\lambda,\mu,\beta}(X) = \gamma X$, $\tau_{\lambda,\mu,\beta}(Y) = \gamma\inv Y$ and $\tau_{\lambda,\mu,\beta}(Z) = Z$. 
Thus, the action of $g$ is diagonal with respect to this presentation
and so by \cite[Theorem 2.7]{JW}, $R^{\grp{g}}=\kk[Z][X^\ell,Y^\ell;\varsigma^\ell, A(Z)]$ where 
$A(Z) = \prod_{i=0}^{n-1} \varsigma^{-i}(a'(Z))$.

Finally, suppose we are in the case $g=\tau_{\lambda,\mu,\beta}\circ\Omega$
but $\lambda\mu=1$. Here, we set $K_+=\rt$, $K_- = 0$, and let
\begin{align*}
X &= \left( \beta-\lambda^2 \right) x
	+ \left( \frac{\lambda^2}{\beta-\lambda^2} \right)y + 2\lambda z 
    - (\rt+1)\lambda, \\
Y &= \left( \frac{1}{\beta-\lambda^2} \right)y, \quad
Z = \left( \frac{\lambda}{\beta-\lambda^2} \right)y + z.
\end{align*}
The same argument as before works with $\gamma=\lambda^2/\beta$.
\end{proof}

The next theorem is analogous to the main result in \cite{Sm},
as well as \cite[Theorem 2]{AP}.
That is, these GWAs are {\it rigid} with respect to cyclic group actions.

\begin{corollary}
\label{cor.smith}
Let $R$ be a classical GWA with $\deg_z(a)=2$ and let $G,H \subset \af(R)$
be finite cyclic groups. If $R^G \iso R^H$, then $G \iso H$. In particular,
if $R^H \iso R$, then $H$ is trivial.
\end{corollary}

\begin{proof}
An isomorphism of classical GWAs must preserve the degree of the 
defining polynomial \cite[Theorem 3.28]{BJ}.
By Theorem \ref{thm.fixed}, the degree of the defining polynomial is
in $R^G$ (resp. $R^H$) is $2|G|$ (resp. $2|H|$).
Thus, if $R^G \iso R^H$, then $|G|=|H|$.
\end{proof}

Let $R=\kk[z][x,y;\sigma,a]$ be a GWA, not necessarily classical.
Two roots $\alpha, \beta$ of $a$ are said to be {\sf congruent} if there exists an $i \in \ZZ$ such that, as ideals of $\kk[z]$, $\left(\sigma^i(z-\alpha)\right) = \left( z - \beta \right)$.
By \cite[Theorem 1.6]{B2} and \cite[Theorem 4.4]{H}, 
the global dimension of $R$ satisfies
\[
\gldim R = \begin{cases}
\infty 	&	\text{if $a$ has a multiple root} \\
2		&	\text{if $a$ has a congruent root and no multiple roots} \\
1		&	\text{if $a$ has no congruent roots and no multiple roots.}
\end{cases}
\]
\begin{corollary}
\label{cor.gldim}
Let $R=\kk[z][x,y;\sigma,a]$ be a classical GWA with $a=z(z-\rt)$
and let $H \subset \af(R)$ be a finite cyclic group with $|H|>2$.
\begin{enumerate}
\item If $\gldim(R)=\infty$, then $\gldim(R^H)=\infty$.
\item If $\gldim(R)=1$, then $\gldim(R^H)=1$.
\item If $\gldim(R)=2$, then $\rt \in \ZZ$ and 
\[
\gldim(R^H)=\begin{cases}
2 & \text{if $|\rt|\geq |H|$} \\
\infty & \text{otherwise}.
\end{cases}
\]
\end{enumerate}
\end{corollary}

\begin{proof}
By Theorem \ref{thm.fixed}, it suffices to consider the fixed ring by a diagonal action on $R$.
Note that in Theorem \ref{thm.fixed}, we have $|K_+-K_-|= \rt$ and so
the change of basis does not affect the difference between the roots.
We freely use the notation from that theorem in this proof.

Recall that $A(Z)=\prod_{i=0}^n \varsigma^{-i}(a'(Z))$.
If $a$ has a multiple root then so does $A$, proving (1). 
Suppose that $\rt >0$. The case $\rt <0$ is similar.
Then the roots of $A(Z)$ are $0,1,\hdots,n-1,\rt ,\rt +1,\hdots,\rt +(n-1)$.
In this case, $\gldim(R^H)=\infty$ if and only if $0 < \rt \leq n-1$.
Because the automorphism associated to $R^H$ is $\varsigma^n$,
where $\varsigma(Z)=Z-1$, then it follows that $R^H$ has congruent
roots if and only if $R$ has congruent roots, proving (2).
Furthermore, if $\gldim(R)=2$, so $\rt \in \ZZ$ and $\rt \neq 0$, then $\gldim(R^H) \geq 2$
and $A(Z)$ has multiple roots if and only if $\rt \geq |H|$.
\end{proof}

It is clear that Corollary \ref{cor.gldim} also applies to higher degree
classical GWAs under the action of $\Theta_\beta$.
As another application of Theorem \ref{thm.fixed}, we consider the
Calabi-Yau property for fixed rings of classical GWAs with $\deg_z(a)=2$.

For an algebra $A$, we denote the {\sf enveloping algebra} of $A$ by
$A^e=A \tensor A^{op}$, where $A^{op}$ is the opposite algebra of $A$.
The algebra $A$ is {\sf homologically smooth} if it it has a projective
resolution of finite length in $A^e$.
If, further, there exists $d \in \NN$ such that
$\Ext_{A^e}^i(A,A^e) \iso \delta_{ij} A$, where $\delta_{ij}$ is the Kronecker-delta function,
%\[ \Ext_{A^e}^i(A,A^e) \iso 
%	\begin{cases}
%   	0 & \text{if $i \neq d$} \\
%        A^\nu & \text{otherwise},
%    \end{cases}\]
then $A$ is said to be {\sf Calabi-Yau of dimension $d$}.

By a result of Liu, a classical GWA $R=\kk[z][x,y;\sigma,a]$
is Calabi-Yau if and only if $R$ has finite global dimension \cite[Theorem 1.1]{L}.
The $\Ext$ condition holds for all classical GWAs, but Liu proves
that the finite global dimension hypothesis implies homological smoothness.
It is noted in the discussion that for a homologically smooth algebra $A$,
the Calabi-Yau dimension is bounded below by the global dimension.
Consequently, if $\gldim(R)=\infty$, then $R$ is not homologically smooth.
The next result now follows from the previous corollary.

\begin{corollary}
\label{cor.CY}
Let $R=\kk[z][x,y;\sigma,a]$ be a classical GWA with $a=z(z-\rt)$
and let $H \subset \af(R)$ be a finite cyclic group with $|H|>2$.
Then $R^H$ is Calabi-Yau if and only if $R$ is Calabi-Yau
and either $\gldim(R)=1$ or $\gldim(R)=2$ and $|\rt| \geq |H|$.
\end{corollary}

In most cases, we are unable to say whether the fixed ring of a classical GWA by a 
non-cyclic group of filtered automorphisms is a GWA. However, we can in one special case.

\begin{corollary}
Let $R=\wa$ and $H=\grp{\Theta_{-1},\Omega} \subset \af(R)$.
Then $R^H$ is a classical GWA.
\end{corollary}

\begin{proof}
By Proposition \ref{prop.wa}, $R^{\grp{\Theta_{-1}}}$ is a classical GWA of degree 2. 
Since $\Omega$ is a filtered automorphism on $R^{\grp{\Theta_{-1}}}$, then
$R^H = \left(R^{\grp{\Theta_{-1}}}\right)^{\grp{\Omega}}$. 
The result now follows from Theorem \ref{thm.fixed}.
\end{proof}

\begin{comment}
To close this section, we make a brief remark regarding Hopf actions on classical GWAs that preserve the filtration.
Let $R$ be a classical GWA and recall that under the standard filtration we have that $\gr(R)$ is a commutative domain.
By \cite[Proposition 5.4]{EWcomm}, if a semisimple Hopf algebra $H$ acts on $R$
inner faithfully and preserves the filtration, then $H$ is a group algebra.

We can now ask whether this result extends to finite dimensional Hopf algebras.
\jason{This is proved for the $n$th Weyl algebra in \cite[Theorem 0.3]{CWWZ}.
Not clear to me yet whether this result adapts for GWAs.}
\end{comment}

\section{The case $n \geq 3$}
Let $R=\kk[z][x,y;\sigma,a]$ with $n = \deg(a) \geq 3$. By Theorem \ref{thm.ngeq3}, if $a$ is not reflective then $\af(R)$ is generated by the maps $\Theta_{\beta}$ and if $a$ is reflective than $\af(R)$ is generated by $\Omega$ and the maps $\Theta_{\beta}$. In the former case, any finite subgroup $H$ of $\af(R)$ will be cyclic, generated by some $\Theta_{\beta}$ where $\beta$ is a root of unity. In this case, by \cite[Theorem 2.7]{JW}, $R^H$ is again a generalized Weyl algebra.

We now study the finite subgroups of $\af(R)$ when $a$ is reflective.
\begin{theorem}
\label{thm.refgrp}
Let $R = \kk[z][x,y;\sigma,a]$ be a classical GWA with $a\in \kk[z]$ a reflective polynomial of degree $n \geq 3$. Let $H$ be a nontrivial finite subgroup of $\af(R)$. If $n$ is even, then one of the following holds:
\begin{enumerate}
\item \label{case.c2} $H = \grp{\Theta_{\beta} \circ \Omega}$ for some $\beta \in \kk$ and $H \cong C_2$, 
\item \label{case.cyc} $H= \grp{\Theta_{\lambda}}$ for some $\lambda \in \kk$ an $m$th root of unity and $H \cong C_m$, or
\item \label{case.dih} $H = \grp{\Theta_{\beta} \circ \Omega, \Theta_{\lambda}}$ for some $\beta, \lambda \in \kk$ with $\lambda$ an $m$th root of unity and $H \cong D_{2m}$, the dihedral group of order $2m$.
\end{enumerate}
If $n$ is odd, then one of the following holds:
\begin{enumerate}
\item \label{case.c4} $H = \grp{\Theta_{\beta} \circ \Omega}$ for some $\beta \in \kk$ and $H \cong C_4$, 
\item \label{case.cyc2} $H= \grp{\Theta_{\lambda}}$ for some $\lambda \in \kk$ an $m$th root of unity and $H \cong C_m$, or
\item \label{case.dih2} $H = \grp{\Theta_{\beta} \circ \Omega, \Theta_{\lambda}}$ for some $\beta, \lambda \in \kk$ with $\lambda$ an $2m$th root of unity and $H \cong \operatorname{Dic}_{m}$, the binary dihedral group of order $4m$.
\end{enumerate}
\end{theorem}
\begin{proof}Using Proposition \ref{prop.gmaps}, each element of $\af(R)$ can be written as either $\Theta_{\beta}$ or $\Theta_{\beta} \circ \Omega$ for some $\beta \in \kk$. The automorphisms $\Theta_{\beta} \circ \Omega$ have order $2$ if $n$ is even and order $4$ if $n$ is odd. The automorphisms $\Theta_{\beta}$ have finite order if and only if $\beta$ is a root of unity.

We consider the case that $n$ is even. The case of $n$ odd is similar. Let $H$ be a finite subgroup of $\af(R)$. Suppose first that $H$ does not contain $\Theta_{\beta}$ for any $\beta \neq 1$. Because $\Theta_{\beta} \circ \Omega \circ \Theta_{\gamma} \circ \Omega = \Theta_{\beta \gamma \inv}$, this means that $H$ is generated by a single $\Theta_{\beta} \circ \Omega$ and we are in case \ref{case.c2}.

Now if $H$ contains $\Theta_{\beta}$, then $\beta$ must be a root of unity. Consider the subgroup of $C$ of $H$ consisting of elements of the form $\Theta_{\beta}$. Since $\Theta_{\beta} \circ \Theta_{\gamma} = \Theta_{\beta \gamma}$, $C$ is generated by a single $\Theta_{\lambda}$ where $\lambda$ is an $m$th root of unity. If $H = C$, then we are in case \ref{case.cyc}.

So now suppose that $H \neq C$, so $H$ contains some $\Theta_{\beta} \circ \Omega$. If $\Theta_{\gamma} \circ \Omega \in H$, then since $\Theta_{\beta} \circ \Omega \circ \Theta_{\gamma} \circ \Omega = \Theta_{\beta \gamma \inv}$, we must have that $\beta \gamma\inv = \lambda^j$ for some $0 \leq j < m$, so $\gamma = \beta \lambda^{-j}$ and hence $\Theta_{\gamma} \circ \Omega = \Theta_{\beta} \circ \Omega \circ \Theta_{\lambda^j}$, and we are in case \ref{case.dih}. Therefore, $H$ is generated by $\Theta_{\beta} \circ \Omega$ and $\Theta_{\lambda}$. Finally, 
\[ \Theta_{\beta} \circ \Omega \circ \Theta_{\lambda} = \Theta_{\lambda}\inv \circ \Theta_{\beta} \circ \Omega
\]
so $H \cong D_{2m}$, as claimed.
\end{proof}

%\rob{So if $a$ is not reflective, it's just Jordan Wells. If $a$ is reflective, then it is either just Jordan Wells, just one $\Omega$, or a dihedral, and we can get the fixed ring by first taking the fixed ring of the cyclic part (by Jordan Wells) and then just doing the one $\Omega$.}

The classical GWA $R = \kk[z][x,y;\sigma, a]$ is naturally $\ZZ$-graded by letting $\deg x = 1$, $\deg y = -1$ and $\deg f = 0$ for all $f \in \kk[z]$. Under this grading, the maps $\Theta_{\beta}$ are graded automorphisms and the map $\Omega$ reverses the grading on $R$. In what follows, we exploit this $\ZZ$-grading on $R$.

\begin{theorem} 
\label{thm.n3omega}
Let $R = \kk[z][x,y; \sigma, a]$ be a classical GWA with $n=\deg_z(a)$. Suppose that $a$ is a reflective polynomial so there exists $\rho \in \kk$ with $a(\rho - z) = (-1)^n a(z)$. If $n$ is even, then $R^{\grp{\Theta_{\beta} \circ \Omega}}$ is generated over $\kk[z(1+\rho-z)]$ by $x+\beta y$ and $zx + \beta (1+\rho-z)y$.
%\rob{The computations below suggest you only need $x+\beta y$?}. 
If $n$ is odd, then $R^{\grp{\Theta_{\beta} \circ \Omega}}$ is generated over $\kk[z(1+\rho-z)]$ by $x^2 + \beta^2 y^2$ and $zx^2 + \beta^2 (1+ \rho - z)y^2$.
\end{theorem}
\begin{proof}By using the relations in $R$, each element $r$ of $R$ can be written 
\[ r = \sum_{i=0}^m f_i(z) x^i + \sum_{i=1}^m g_i(z) y^i
\]
for some $f_i(z),g_i(z) \in \kk[z]$. Since $\Theta_{\beta} \circ \Omega$ reverses the $\ZZ$-grading on $R$, if $r$ is fixed by $\Theta_{\beta} \circ \Omega$, then we must have, for each $1 \leq i \leq m$,
\[ \Theta_{\beta} \circ \Omega(f_i(z)x^i) = g_i(z)y^i \quad \text{and} \quad \Theta_{\beta} \circ\Omega(g_i(z)y^i) = f_i(z)x^i
\]
and hence
\[ \beta^{-i} f_i(1 + \rho -z) = g_i(z) \quad \text{and} \quad (-1)^{ni} \beta^{-i} f_i(1+ \rho -z) = g_i(z).
\]
Hence, the only nonzero summands of $r$ occur when $ni$ is even.

In particular, when $i = 0$, we must have that $f_0(1 + \rho -z) = f_0(z)$. By an induction argument, each $g(z) \in \kk[z]$ can be written as $h_1 + z h_2$ for some $h_1, h_2 \in \kk[z(1+\rho-z)]$. Hence, if $g(1+\rho-z) = g(z)$, then $h_2 = 0$ so $f_0 \in \kk[z(1+\rho-z)]$.

If $n$ is even, then each invariant is a sum of terms of the form $f(z)x^{m} + \beta^{m} f(1+\rho - z) y^{m}$ where $m \geq 0$ and $f(z) \in \kk[z]$. We claim that each of these elements is generated over $\kk[z(1+\rho-z)]$ by $x+\beta y$ and $zx + \beta (1+\rho-z)y$. Since each $f(z) \in \kk[z]$ can be written as $h_1 + z h_2$ for some $h_1, h_2 \in \kk[z(1+\rho-z)]$, therefore we can write any 
\[ f(z) x^j + \beta^j f(1+\rho - z) y^j
\]
as a $\kk[z(1+\rho-z)]$-linear combination of $x^j+\beta^j y^j$ and $zx^j + \beta^j(1+\rho-z)y^j$.

It therefore suffices to show that for any $j \geq 0$, we can generate any $x^j + \beta^j y^j$ and $z x^j + \beta^j(1+\rho-z)y^j$. Now observe that 
\[ \left( x^{j-1} + \beta^{j-1}y^{j-1}\right) (x+\beta y) = x^j + \beta x^{j-1}y + \beta^{j-1} y^{j-1}x + \beta^{j} y^{j}
\]
and since $a(z) = a(\rho-z)$,
\begin{align*} 
\beta x^{j-1}y + \beta^{j-1} y^{j-1}x 
&= \beta x^{j-2}a(z-1)  +\beta^{j-1}  y^{j-2}a(z) \\
&= \beta a(z-j+1)x^{j-2}+ \beta^{j-1} a(z+j-2)y^{j-2} \\
&= \beta a(z-j+1)x^{j-2} + \beta^{j-1} a(\rho - (z +j-2)) y^{j-2} \\
&= \beta \left[ a(z-j+1)x^{j-2} +  \beta^{j-2} a((1 + \rho - z) -j + 1)) y^{j-2} \right]
\end{align*}
so by induction we can generate any $x^j + \beta^{j} y^j$. By a similar argument, we can generate any $z x^j + \beta^{j} (1+\rho-z)y^j$. Therefore, the invariant ring has the claimed generators. The proof when $n$ is odd is similar.
\end{proof}

\begin{corollary}
\label{cor.n3omega}
Let $R = \kk[z][x,y; \sigma, a]$ be a classical GWA with $n=\deg_z(a)$ and $a$ reflective. 
Let
\[
A = \begin{cases}
x+\beta y		&	\text{$n$ even} \\
x^2+\beta^2 y^2	&	\text{$n$ odd,}
\end{cases}
\qquad
B = \begin{cases}
zx+\beta (1+\rho-z)y		&	\text{$n$ even} \\
zx^2+\beta^2(1+\rho-z)y^2	&	\text{$n$ odd,}
\end{cases}
\]
and $C = z(1 + \rho - z)$ so that $A$, $B$, and $C$ generate $R^{\grp{\Omega}}$.
If $n$ is even, then the generators satisfy the following relations
\begin{align*}
[A,C] &= 2B-(2+\rho)A, &	[B,A] &= A^2 + \beta f(C), \\
[B,C] &= \rho B-2CA, &	B^2 &= \rho BA-CA^2 + \beta g(C).
\end{align*}
If $n$ is odd, then the generators satisfy the following relations
\begin{align*}
[A,C] &= 4B-2(3+\rho)A, 	&	[B,A] &= 2A^2 + \beta f(C), \\
[B,C] &= 2(\rho-1)B-4CA, 	&	B^2 &= (\rho-1)BA-CA^2 + \beta g(C).
\end{align*}
In both cases, $f(C)$ and $g(C)$ represent polynomials in $C$ with
\[
\deg_C(f) = \begin{cases}
n	&	\text{$n$ even} \\
2n	&	\text{$n$ odd,}
\end{cases}
\quad\text{and}\quad
\deg_C(g) = \begin{cases}
1+\frac{n}{2}	&	\text{$n$ even} \\
2n+1			&	\text{$n$ odd.}
\end{cases}
\]
\end{corollary}

\begin{proof}
Assume $n$ is even. The case of $n$ odd is similar.
This is largely direct computation and we omit those for $[A,C]$ and $[B,C]$. Next we have,
\begin{align*}
[B,A] &= [zx+\beta(1+\rho-z)y,x+\beta y] \\
	&= [z,x]x + \beta[1+\rho-z,y]y + \beta[zx,y] + \beta^2[(1+\rho-z)y,x] \\
    &= (x^2+\beta^2 y^2) + \beta(\rho-2z)yx + \beta(2z-2+\rho)xy \\
    &= A^2 + \beta\left((\rho-1-2z)a + (2z-3-\rho)\sigma(a)\right).
\end{align*}
Observe that $(\rho-1-2z)a + (2z-3-\rho)\sigma(a) \in \kk[z]$ and
\[ \Omega((\rho-1-2z)a + (2z-3-\rho)\sigma(a))
	= (2z-3-\rho)\sigma(a)+(\rho-1-2z)a.
\]
Hence, it must be possible to express this as a polynomial in $C$.   
Finally we have
\[
B^2 = \rho BA-CA^2 + \beta((3+\rho)z-2z^2)\sigma(a) + \beta((1-\rho^2)+(3\rho +1)z-2z^2)a.
%&= z(z-1)x^2 + \beta(z(2+\rho-z)xy + (1+\rho-z)(z+1)yx) + \beta^2 (1+\rho-z)(\rho-z)y^2 \\
\]
As in the computation for $[B,A]$, the remaining polynomial in $z$ is fixed by $\Omega$.
\end{proof}

It is not clear to us whether $R^{\grp{\Omega}}$ is a GWA for $\deg_z(a) \geq 3$.
One piece of evidence against is the following.
A classical GWA $R=\kk[z][x,y;\sigma,a]$ 
with $\deg_z(a)=2$ can be presented with two generators
by solving the relation $yx-xy=a-\sigma(a)$ for $z$ and substituting into the other relations.
When $\deg_z(a)>2$, one cannot generate $R$ using only $x$ and $y$,
but whether one can use a different pair of generators for $R$ is unclear.
%\rob{I wonder if the following sentence is too strong. It is clear that you cannot generate with just $x$ and $y$ if $\deg a > 2$, but I'm less sure about no natural way.}
%It is clear that there is no natural way to do this when $\deg_z(a) > 2$.
We would expect that, were $R^{\grp{\Omega}}$ to be a classical GWA for $\deg_z(a) \geq 3$, 
then the degree of the corresponding defining polynomial would be higher and thus
not able to be presented with two generators.
However, one observes from Corollary \ref{cor.n3omega} that it is possible to take
\[
B = \begin{cases}
\frac{1}{2}\left( AC-CA + (2+\rho)A\right)	&	\text{$n$ even} \\
\frac{1}{4}\left( AC-CA + 2(3+\rho)A\right)	&	\text{$n$ odd}.
\end{cases}
\]

\section{Auslander's Theorem}
\label{sec.auslander}

In this final section we consider Auslander's theorem.
As stated in the introduction, it is sufficient in many cases to show that $\p(A,G) \geq 2$
for an algebra $A$ and a group $G$ acting on $A$.
In particular, by various results in \cite{BHZ2,BHZ1}, this applies when
\begin{enumerate}
\item $A$ is noetherian, connected graded AS regular, and Cohen-Macaulay of GK dimension at least two, and $G$ is a group acting linearly on $A$;
\item $A$ is a noetherian PI and Kdim-CM algebra of Krull dimension at least 2;
\item $A$ is {\it congenial} and $G$ preserves the filtration on $A$.
\end{enumerate}

Our focus will be on the last condition.
We refer to \cite{BHZ2} for a full definition.

\begin{lemma}
\label{lem.cntp}
Suppose $F$ is a field of characteristic $p > 0$ and $R=F[z][x,y;\sigma,a]$ a classical GWA.
Then $F[x^p,y^p] \subset \cnt(R)$.
\end{lemma}

\begin{proof}
It is clear that $\sigma^p=\id$ and,
moreover, $\sigma^k=\id$ if and only if $p \mid k$. Thus, $[x^p,z] = [y^p,z]=0$.
Since $xy-yx = a-\sigma\inv(a)$, then it follows by induction that
\begin{align*}
	x^ky-yx^k &= (\sigma^{k-1}(a)-\sigma\inv(a))x^{k-1}, \\
	xy^k-y^kx &= (a-\sigma^{-k}(a))y^{k-1}.
\end{align*}
Setting $k=p$ gives $[x^p,y]=[x,y^p]=0$ and the claim holds.
\end{proof}

Let $R=\kk[z][x,y;\sigma,a]$ be a classical GWA and write
\[ a = z^n + c_{n-1}z^{n-1} + \cdots + c_1z + c_0, \quad c_i \in \kk.\]
Set $D=\ZZ[c_0,\hdots,c_{n-1}]$, then it is not difficult to see that
$R_D=D[z][x,y;\sigma,a]$ is again a GWA.
Moreover, $R_D$ is free over $D$ with a basis consisting of the standard monomials and 
\[
R_D = D[z][x,y;\sigma,a] \tensor_D \kk 
= (D \tensor_D \kk)[z][x,y;\sigma,a]
= \kk[z][x,y;\sigma,a] = R.
\]
That is, $R_D$ is an {\it order} of $R$.
Next, we check the conditions of congeniality.
\begin{enumerate}
\item Under the standard filtration, $R$ is a noetherian locally finite filtered algebra with the standard filtration.
\item The algebra $R_D$ is also noetherian locally finite filtered (over $D$)
and the standard filtration on $R$ induces a filtration on $R_D$.
\item It is clear that $\gr R_D$ is an order of $\gr R$.
\item It is well-known that $\gr R_D = D[x,y,z]/(xy-z^n)$ is strongly noetherian
and a locally finite graded algebra over $D$.
\item Let $F$ be a factor ring of $D$ that is a finite field of characteristic $p$.
Then \[ R_D \tensor_D F \iso (D \tensor_D F)[z][x,y;\sigma,a]\] 
and hence $R_D \tensor_D F$ is noetherian.
Moreover, by Lemma \ref{lem.cntp}, it is module finite over the commutative subalgebra
$\kk[x^p,y^p,z,z^2,\hdots,z^{n-1}]$.
\end{enumerate}

We now adapt the methods of \cite{GKMW} to show that the pertinency
condition is satisfied for a classical GWA and a cyclic subgroup of filtered automorphisms.

\begin{lemma}
\label{lem.auslander1}
Let $R=\kk[z][x,y;\sigma,a]$ be a classical GWA.
Set $G=\grp{\Theta_\beta}$ with $\beta$ a primitive $\ell$th root of unity,
$\ell \geq 2$.
Then the Auslander map is an isomorphism for the pair $(R,G)$.
\end{lemma}

\begin{proof}
Set $S=\gr(R)$ under the standard filtration.
Then $G$ acts as graded automorphisms on $S$.
We claim first that the theorem holds for the pair $(S,G)$. Define 
\[f = \sum_{i=0}^{\ell-1} 1 \# (\Theta_\beta)^i \in S \# G.\]
Now observe that
\begin{align*}
xf - f(\beta x) 
	= \sum_{i=0}^{\ell-2} (1-\beta^{i+1}) x \# (\Theta_\beta)^i
    \in (f).
\end{align*}
Repeating this process we find that $x^{\ell-1}\# e \in (f)$.
Similarly we can show that $y^{\ell-1}\# e \in (f)$ and so 
\[ y^{\ell-1}x^{\ell-1} = (a_2)^{\ell-1} \in (f).\]
Through the natural embedding $S \hookrightarrow S\# G$ given by $s \mapsto s\#e$,
we have \[\GKdim S \# G /(f) = \GKdim S/ ((f) \cap S).\]
It follows from the above computation that $S/ ((f) \cap S)$ is finite-dimensional 
and so $\p(S,G) = 2$. Thus, the Auslander map is an isomorphism for $(S,G)$ \cite[Theorem 0.2]{BHZ2}.
%\rob{Probably worth mentioning that $\GKdim S \# G /(f) = \GKdim S/ ((f) \cap S)$, which is really what we're invoking here. Anyway, I agree with the pertinency computation that $\p (R,G)= \p(S,G) = 2$. Whether this means the Auslander map is an isomorphism I'm not as sure...}

The action of $\Theta_\beta$ respects the standard filtration on $R$,
both $R$ and $S$ are noetherian, and as $S$ is a commutative complete intersection ring,
it is CM and thus $R$ is CM by \cite[Lemma 4.4]{SZ}.
Hence, $\p(R,G) \geq \p(S,G)=2$ by \cite[Proposition 3.6]{BHZ2} and so
the theorem holds for $(R,G)$ by \cite[Theorem 3.3]{BHZ2}.
\end{proof}

\begin{lemma}
\label{lem.auslander2}
Let $R = \kk[z][x,y;\sigma,a]$ be a classical GWA with $a\in \kk[z]$ reflective,
$\deg_z(a) \geq 3$, and $\beta \in \kk^\times$.
The Auslander map is an isomorphism for the pair $(R,\grp{\Theta_\beta \circ \Omega})$.
\end{lemma}

\begin{proof}
This follows similarly to Theorem \ref{thm.auslander}.
Set $\phi=\Theta_\beta \circ \Omega$ and $H=\grp{\phi}$.
Throughout, let $S=\gr(R)$ and $f = 1\#e + 1\#\phi \in S\#H$.

First we consider the case of $n$ odd. 
We have $xf+f\beta\inv y = (x+\beta\inv y)\#e \in (f)$. On the other hand, $yf-f(\beta x) = (y-\beta x)\#e \in (f)$.
It follows that $x\#e,y\#e \in (f)$.

Next we suppose $n$ is even. Then $xf-f(\beta\inv y) = (x-\beta\inv y)\#e \in (f)$.
Similarly, $(x^2-\beta^{-2} y^2)\#e \in (f)$ and $zf+fz = 2z\#e \in (f)$.
It now follows that $x^2\#e,y^2\#e \in (f)$.

Hence, in either case, we have $S/((f) \cap S)$ is finite-dimensional and so $\p(R,H)\geq \p(S,H)=2$.
\end{proof}

\begin{theorem}
\label{thm.auslander}
Let $R=\kk[z][x,y;\sigma,a]$ be a classical GWA
and let $G$ be a finite nontrivial cyclic subgroup of $\af(R)$.
Then the Auslander map is an isomorphism for the pair $(R,G)$.
\end{theorem}

\begin{proof}
The case $\deg_z(a)=1$ is a consequence of \cite[Theorem 2.4]{Mo}.
If $\deg_z(a)=2$, then we apply the change of basis in Theorem \ref{thm.fixed}
and the result follows from Theorem \ref{lem.auslander1}.
Finally, if $\deg_z(a)>2$, then by Theorem \ref{thm.ngeq3},
$G=\grp{\Theta_\beta}$ or $G=\grp{\Theta_\beta \circ \Omega}$
and so the result follows from Lemma \ref{lem.auslander1} and Lemma \ref{lem.auslander2}.
\end{proof}

We end with a brief remark on the structure of the skew group ring appearing in
the above results.
Let $R=\kk[z][x,y;\sigma,a]$ be a classical GWA and let $G=\grp{\Theta_\beta}$, 
$2 \leq |\beta| < \infty$.
Then $R\# G \iso RG[x,y;\hat{\sigma},\hat{a}]$ where $RG$ is the group algebra of $G$
with coefficients in $R$ and $\hat{\sigma},\hat{a}$ are naturally extended to $RG$ from $R$.
That is, $\hat{a}=a\# e$, and 
$\hat{\sigma}(p \# \Theta_\beta^k) = \beta^{-k}(\sigma(p) \# \Theta_\beta^k)$.
When $\deg_z(a)=2$, one can apply the change of basis in Theorem
\ref{thm.fixed} and achieve the same result for any finite cyclic 
group acting linearly on $R$.
Theorem \ref{thm.auslander} now implies, by way of the Auslander map, that 
the corresponding endomorphism ring has the structure of a GWA .

%\subsection*{Acknowledgment}

\bibliographystyle{plain}
%\bibliography{gwabib}{}

\end{document}